\documentclass[11pt,a4paper]{amsart}

\usepackage{amsfonts,amsmath,amssymb,amsthm}

\usepackage[utf8]{inputenc}
\usepackage[T1]{fontenc}

\usepackage{float} 
\usepackage{hyperref} 
\usepackage{stmaryrd}
\usepackage{siunitx}

\newcommand{\Zz}{\mathbb{Z}}
\newcommand{\Qq}{\mathbb{Q}}

\renewcommand {\leq}{\leqslant}
\renewcommand {\geq}{\geqslant}
\renewcommand {\le}{\leqslant}
\renewcommand {\ge}{\geqslant}
\renewcommand {\epsilon}{\varepsilon}

\newcommand{\divides}{\mathrel{\mid}}
\newcommand{\notdivides}{\mathrel{\not|}}

\newcommand{\defi}[1]{\emph{\textbf{#1}}}

\newcommand{\Tree}{\mathop{\mathrm{Tree}}\nolimits}
\newcommand{\Trunk}{\mathop{\mathrm{Trunk}}\nolimits}
\newcommand{\Fan}{\mathop{\mathrm{Fan}}\nolimits}
\newcommand{\val}{\mathop{\mathrm{val}}\nolimits}
\newcommand{\mult}{\mathop{\mathrm{mult}}\nolimits}

\makeatletter
\renewcommand{\pod}[1]{\allowbreak\mathchoice
	{\if@display \mkern 8mu\else \mkern 8mu\fi (#1)}
	{\if@display \mkern 8mu\else \mkern 8mu\fi (#1)}
	{\mkern4mu(#1)}
	{\mkern4mu(#1)}
}
\makeatother

\makeatletter
\newsavebox\myboxA
\newsavebox\myboxB
\newlength\mylenA

\newcommand*\xoverline[2][0.75]{%
	\sbox{\myboxA}{$\m@th#2$}%
	\setbox\myboxB\null
	\ht\myboxB=\ht\myboxA%
	\dp\myboxB=\dp\myboxA%
	\wd\myboxB=#1\wd\myboxA
	\sbox\myboxB{$\m@th\overline{\copy\myboxB}$}
	\setlength\mylenA{\the\wd\myboxA}
	\addtolength\mylenA{-\the\wd\myboxB}%
	\ifdim\wd\myboxB<\wd\myboxA%
	\rlap{\hskip 0.5\mylenA\usebox\myboxB}{\usebox\myboxA}%
	\else
	\hskip -0.5\mylenA\rlap{\usebox\myboxA}{\hskip 0.5\mylenA\usebox\myboxB}%
	\fi}
\makeatother

\newcommand{\reduc}[1]{\xoverline{#1}}

\makeatletter
\def\subsection{\@startsection{subsection}{2}%
	\z@{.5\linespacing\@plus.7\linespacing}{.3\linespacing}%
	{\normalfont\bfseries}}
\makeatother

{\theoremstyle{plain}
	\newtheorem{theorem}{Theorem}[section]    
	\newtheorem{lemma}[theorem]{Lemma}       
	\newtheorem{proposition}[theorem]{Proposition}      
	\newtheorem{corollary}[theorem]{Corollary}      
	\newtheorem*{theorem*}{Theorem}
}
{\theoremstyle{remark}
	\newtheorem{definition}[theorem]{Definition}      
	\newtheorem*{remark*}{Remark}  
	\newtheorem{example}[theorem]{Example}
}

\usepackage{tikz} 
\usetikzlibrary{calc}
\usetikzlibrary{decorations.pathreplacing}

\newcommand{\myfigure}[2]{
	\begin{center}\small
		\tikzstyle{every picture}=[scale=1.0*#1]
		#2
\end{center}}

\usepackage[a4paper]{geometry}
\title{Solutions of a polynomial equation modulo a prime power}

\author{Arnaud Bodin}
\author{Christian Drouin}

\email{arnaud.bodin@univ-lille.fr}
\email{christian.drouin@wanadoo.fr}

\address{Université de Lille, CNRS, Laboratoire Paul Painlevé, 59000 Lille, France}
\address{Seignosse, France}

\subjclass[2020] {Primary 11A07; Sec. 11D45, 11S05}

\keywords{polynomial, congruence, tree}

\date{\today}

\begin{document}

\begin{abstract}
How do you find the integer solutions of a polynomial equation modulo an integer?
\end{abstract}

\maketitle


\section{Introduction}

\subsection{Roots of polynomials over $\Zz/n\Zz$}

If $p$ is a prime number, the ring $\Zz/p\Zz$ is actually a field.
Thus, a polynomial $P(X) \in \Zz[X]$ of degree $d$ has at most $d$ roots in $\Zz/p\Zz$.
Problems arise when calculations are done modulo an arbitrary integer $n$.
For example, what are the solutions to the equation
\[
x^2+11 \equiv 0 \pmod{15}?
\]
There are $4$ solutions $\{2, 7, 8, 13\}$ even though the equation is indeed a polynomial equation of degree $2$.

\medskip

Even very simple equations can have surprisingly many solutions. 
For instance, take $P(X) = X^2$.  
When working modulo $p^{2e}$ with $p>2$, the equation
\[
x^2 \equiv 0 \pmod{p^{2e}}
\]
has not two but $p^e$ distinct solutions:
\[
x_i = i p^e \quad \text{for } i=0,1,\dots,p^e-1.
\]
Thus, even a degree $2$ polynomial can have exponentially many solutions as the modulus grows. 

\medskip

Finally, Shamir \cite{Shamir1993} gave the remarkable example of the polynomial $P(X) = X$, 
which factors in a surprising way modulo a composite number $n=pq$ with two distinct primes:
\[
X \equiv (p^2+q^2)^{-1}(pX+q)(qX+p) \pmod{pq},
\]
where $p^2+q^2$ is invertible modulo $n$, and $pX+q$ and $qX+p$ are irreducible over $\Zz/n\Zz$.  
Even such a simple polynomial can behave in subtle ways when the modulus is not prime.

\subsection{Reduction to a prime power modulus}

How should one understand these phenomena? 
If $p$ is prime, the ring $\Zz/p\Zz$ is a field, so a polynomial of degree $d$ has at most $d$ roots modulo $p$.  

If $n = \prod_{i=1}^{l} p_i^{e_i}$ is the prime factorization, then solving $P(x) \equiv 0 \pmod{n}$ is equivalent to
solving $P(x) \equiv 0 \pmod{p_i^{e_i}}$ for each $i=1,\ldots,l$.
This reduction follows from the Chinese Remainder Theorem, which also provides an efficient way to recombine the solutions modulo each prime power into solutions modulo $n$, using only modular inverses.

But is the problem of determining the roots of a polynomial simpler if the modulus is just a power of a prime number? In fact, no! For example, the polynomial $X^2$ of degree $2$ already has $3$ roots $\{0,3,6\}$ modulo $3^2$. Thus, the real source of complications is already present when $n = p^e$ is a prime power.

\begin{quote}
	\emph{	
		Let $p$ be a prime number, $e\ge0$ an integer, and $P(X) \in \Zz[X]$.
		The purpose of this article is to calculate the integer solutions $x$ of the equation $P(x) \equiv 0 \pmod{p^e}$, and to understand how these solutions evolve as $e$ grows.
	}
\end{quote}

\subsection{Outline}

In this note, we explain how the solutions of the equations $P(x) \equiv 0 \pmod{p^e}$ evolve as $e$ grows, and how they can all be represented in the form of a \emph{tree}. 
Each vertex corresponds to a solution modulo $p^e$, and its children are the solutions of the same equation modulo $p^{e+1}$ that reduce to it modulo $p^e$.
The figure below represents the set of solutions to the equations $P(x) \equiv 0 \pmod{p^e}$ for  
$P(X) = (X^2+3)(X^2+3X+9)$, with $p=3$, for different values of $e$.  
(This example will be revisited later, see Examples \ref{ex:ex1bis} and \ref{ex:ex1}.)

Since this tree can have many vertices, our goal is to concentrate all this information into a much smaller subtree, the \emph{trunk}.
(In our example, this corresponds to the subtree with two edges in red, drawn with thick lines in the figure.)
To each vertex of the trunk, we attach an integer called the \emph{thickness}.


The trunk allows the complete reconstruction of the solution tree (Theorem \ref{th:trunktotree}):  
starting from each vertex of the trunk, we build a \emph{fan} of solutions emerging from this vertex.
In the figure below there are two fans:
the first one consists of all possible children of the vertex $0$ at level $1$, up to level $t_1 = 3$ ($t_1$ being the thickness at this vertex of level $1$). 
The second fan starts at the vertex $3$ at level $2$, up to level $t_1 + t_2 = 4$ ($t_2$ being the thickness at this vertex of level $2$).  
Alternatively, one could simply count the number of solutions at each level $p^e$ without enumerating them (Corollary \ref{cor:count}).

\begin{figure}[H]
	\myfigure{1}{
\begin{tikzpicture}[scale=1.7]
\usetikzlibrary{decorations.pathreplacing}
\tikzset{
  line/.style = {
  },
  vector/.style = {
    thick,-latex
  },
  dot/.style = {
    insert path={
      node[scale=3]{.}
    }
  },
  smalldot/.style = {
    insert path={
      node[scale=1.5]{.}
    }
  }
}

\begin{scope}[xshift=6cm]
 \path
   (0,0) coordinate (O)
 ;

\foreach \i in {0,...,2} {
     \path (1*\i-1,1) coordinate (P\i1);
  }

\foreach \i in {0,...,8} {
     \path ({(2.5*\i-5.5)/3},2) coordinate (P\i2);
  }
\foreach \i in {0,...,26} {
     \path ({(2.5*\i-19)/9},3) coordinate (P\i3);
  }
\foreach \i in {0,...,80} {
     \path ({(3*\i-66)/27},4) coordinate (P\i4);
  }

\path (O) edge[line, very thick, red] (P01);
 
\path (P01) edge[line, blue] (P02);
\path (P01) edge[line, blue] (P22);

\path (P01) edge[line, very thick, red] (P12);

\path (P02) edge[line, blue] (P03);
\path (P02) edge[line, blue] (P13);
\path (P02) edge[line, blue] (P23);

\path (P12) edge[line, blue] (P33);
\path (P12) edge[line, blue] (P43);
\path (P12) edge[line, blue] (P53);

\path (P22) edge[line, blue] (P63);
\path (P22) edge[line, blue] (P73);
\path (P22) edge[line, blue] (P83);

\path (P33) edge[line, thin, blue] (P94);
\path (P33) edge[line, thin, blue] (P104);
\path (P33) edge[line, thin, blue] (P114);

\path (P43) edge[line, thin, blue] (P124);
\path (P43) edge[line, thin, blue] (P134);
\path (P43) edge[line, thin, blue] (P144);

\path (P53) edge[line, thin, blue] (P154);
\path (P53) edge[line, thin, blue] (P164);
\path (P53) edge[line, thin, blue] (P174);


  \foreach \i in {0,...,0} {
     \path  (P\i1) [dot] {};
  }
  \foreach \i in {0,...,2} {
     \path  (P\i2) [dot] {};
  }
  \foreach \i in {0,...,8} {
     \path  (P\i3) [dot] {};
  }
  \foreach \i in {9,...,17} {
     \path  (P\i4) [smalldot] {};
  }
 \path
   (O) [red, dot, red] {}
   (P01) [dot, red] node[right=0.2em,scale=0.7]{$t_1 = 3$}
   (P12) [dot, red] node[right=0.1em,scale=0.7]{$t_2 = 1$}
 ;

 \path
   (P01)  node[left,scale=0.7]{$0$}

   (P02)  node[left,scale=0.7]{$0$}
   (P12) node[left,scale=0.7]{$3$}
   (P22) node[left,scale=0.7]{$6$}

   (P03) node[left,scale=0.6]{$0$}
   (P13) node[left,scale=0.6]{$9$}
   (P23) node[left,scale=0.6]{$18$}
   (P33) node[left,scale=0.6]{$3$}
   (P43) node[left,scale=0.6]{$12$}
   (P53) node[left,scale=0.6]{$21$}
   (P63) node[left,scale=0.6]{$6$}
   (P73) node[left,scale=0.6]{$15$}
   (P83) node[left,scale=0.6]{$24$}

   (P94) node[above,scale=0.5]{$3$}
   (P104) node[above=5pt,scale=0.5]{$30$}
   (P114) node[above,scale=0.5]{$57$}
   (P124) node[above=5pt,scale=0.5]{$12$}
   (P134) node[above,scale=0.5]{$39$}
   (P144) node[above=5pt,scale=0.5]{$66$}
   (P154) node[above,scale=0.5]{$21$}
   (P164) node[above=5pt,scale=0.5]{$48$}
   (P174) node[above,scale=0.5]{$75$}

 ;

  \draw[line,thin,gray] (1,1) -- ++(1,0) node[right, black, fill=white,]{solutions mod $p$};
  \draw[line,thin,gray] (1,2) -- ++(1,0) node[right, black, fill=white, ]{solutions mod $p^2$};
  \draw[line,thin,gray] (1,3) -- ++(1,0) node[right, black, fill=white, ]{solutions mod $p^3$};
  \draw[line,thin,gray] (1,4) -- ++(1,0) node[right, black, fill=white, ]{solutions mod $p^4$};


\end{scope}

\end{tikzpicture}%

	}
	\caption{The trunk and the tree of solutions of $P(X) = (X^2+3)(X^2+3X+9)$, $p=3$.}
	\label{fig:ex0}
\end{figure}
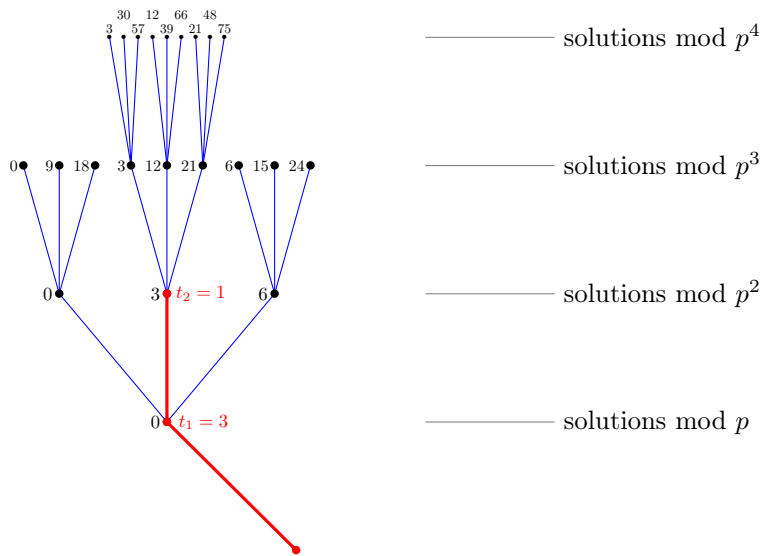

\section{Trees and trunks}
\label{sec:definitions}

\subsection{$p$-adic congruence tree}

We fix $p \ge 2$ a prime number.
The \defi{$p$-adic congruence tree}, denoted $\Omega_p$, is an infinite tree whose root is the pair
$(0, 0)$ of level $e = 0$ and whose vertices of level $e$ are the pairs $(x, e)$, where the integer $x$ is between $0$ and $p^e-1$. The edges of this tree
are those connecting two vertices $(x,e)$ and $(x',e+1)$ such that $x' \equiv x \pmod{p^e}$.
Thus, we can define a partial order relation on the vertices of this graph by: $(x,e) \vartriangleleft (x',e')$
if and only if $e \le e'$ and $x' \equiv x \pmod{p^e}$. In other words, $(x,e)$ is located on the path connecting the root to $(x',e')$.

\begin{figure}[H]
\myfigure{0.9}{
\begin{tikzpicture}[scale=2]

\tikzset{
  line/.style = {
  },
  vector/.style = {
    thick,-latex
  },
  dot/.style = {
    insert path={
      node[scale=3]{.}
    }
  }
}

\def\p{3}
 \path
   (0,0) coordinate (O)
 ;

\foreach \i in {0,...,2} {
     \path (\i-1,1) coordinate (P\i1);
  }

\foreach \i in {0,...,8} {
     \path ({(\i-4)/3},2) coordinate (P\i2);
  }
\foreach \i in {0,...,26} {
     \path ({(\i-13)/9},3) coordinate (P\i3);
  }

  \foreach \i in {0,1,2} {
     \path (O) edge[line] (P\i1);
  }
  \foreach \i in {0,1,2} {
     \foreach \k in {0,1,2} {
         \pgfmathtruncatemacro\ii{\i+3*\k}
          \path (P\k1) edge[line] (P\ii2);
     }
  }
  \foreach \i in {0,1,2} {
     \foreach \k in {0,1,2} {
         \pgfmathtruncatemacro\ii{\i+3*\k}
          \path (P\k1) edge[line] (P\ii2);
     }
  }

  \foreach \i in {0,1,2} {
     \foreach \k in {0,...,8} {
         \pgfmathtruncatemacro\ii{\i+3*\k}
          \path (P\k2) edge[line,thin,gray] (P\ii3);
     }
  }

 \path
   (O) [dot] node[below left]{root}
 ;
  \foreach \i in {0,...,2} {
     \path  (P\i1) [dot] node[right, scale=0.8]{$\i$};
  }
  \foreach \i in {0,...,8} {
     \path  (P\i2) [dot] {};
  }
  \foreach \k in {0,...,2} {
     \pgfmathtruncatemacro\i{0+\k}
     \pgfmathtruncatemacro\ii{0+3*\k}
     \path (P\i2) node[right, scale=0.8]{$\ii$};

     \pgfmathtruncatemacro\i{3+\k}
     \pgfmathtruncatemacro\ii{1+3*\k}
     \path (P\i2) node[right, scale=0.8]{$\ii$};

     \pgfmathtruncatemacro\i{6+\k}
     \pgfmathtruncatemacro\ii{2+3*\k}
     \path (P\i2) node[right, scale=0.8]{$\ii$};
  }
 \path
   (P03) [dot] node[scale=0.5,above]{$0$}
   (P13) [dot] node[scale=0.5,above]{$9$}
   (P23) [dot] node[scale=0.5,above]{$18$}
   (P33) [dot] node[scale=0.5,above]{$3$}
   (P43) [dot] node[scale=0.5,above]{$12$}
   (P53) [dot] node[scale=0.5,above]{$21$}
 ;
  \draw[line,thin,gray] (2,1) -- ++(1,0) node[midway, black, fill=white]{$p$};
  \draw[line,thin,gray] (2,2) -- ++(1,0) node[midway, black, fill=white]{$p^2$};
  \draw[line,thin,gray] (2,3) -- ++(1,0) node[midway, black, fill=white]{$p^3$};

\node at (0,-0.5) {};

\end{tikzpicture}%
\qquad
\begin{tikzpicture}[scale=2]

\tikzset{
  line/.style = {
  },
  vector/.style = {
    thick,-latex
  },
  dot/.style = {
    insert path={
      node[scale=3]{.}
    }
  }
}

\def\p{3}
 \path
   (0,0) coordinate (O)
 ;

\foreach \i in {0,...,2} {
     \path (\i-1,1) coordinate (P\i1);
  }

\foreach \i in {0,...,8} {
     \path ({(\i-4)/3},2) coordinate (P\i2);
  }
\foreach \i in {0,...,26} {
     \path ({(\i-13)/9},3) coordinate (P\i3);
  }

  \foreach \i in {0,1,2} {
     \path (O) edge[line] (P\i1);
     \node[left,blue,scale=0.8] at ($(O)!0.7!(P\i1)$) {$\i$};
  }
  \foreach \i in {0,1,2} {
     \foreach \k in {0,1,2} {
         \pgfmathtruncatemacro\ii{\i+3*\k}
          \path (P\k1) edge[line] (P\ii2);
          \node[left=-0.1em,blue,scale=0.8] at ($(P\k1)!0.7!(P\ii2)$) {$\i$};
     }
  }

  \foreach \i in {0,1,2} {
     \foreach \k in {6,...,8} {
         \pgfmathtruncatemacro\ii{\i+3*\k}
          \path (P\k2) edge[line] (P\ii3);
          \node[blue,left=-0.2em,scale=0.6] at ($(P\k2)!0.7!(P\ii3)$) {$\i$};
     }
  }

\path (O) edge[line,blue,thick] (P21);
\path (P21) edge[line,blue,thick] (P62);
\path (P62) edge[line,blue,thick] (P193);


 \path
   (O) [dot] node[below]{}
 ;
  \foreach \i in {0,...,2} {
     \path  (P\i1) [dot] {};
  }
  \foreach \i in {0,...,8} {
     \path  (P\i2) [dot] {};
  }

 \path(P193) [dot] node[above,scale=0.9]{$\color{red}11$};
\path (0,-0.5) node[above,scale=1]{${\color{red}11} = {\color{blue}2} + {\color{blue}0}p+ {\color{blue}1}p^2$}
;
 ;

\end{tikzpicture}%

}
\caption{Here $p=3$. Left: the $p$-adic congruence tree $\Omega_p$, each vertex is labeled by in integer $x \in[0,p^e-1]$. Right: the decomposition of $x=11$ in base $p$, each edge is labeled by a integer $a_i \in[0,p-1]$.}
\label{fig:padictree}
\end{figure}
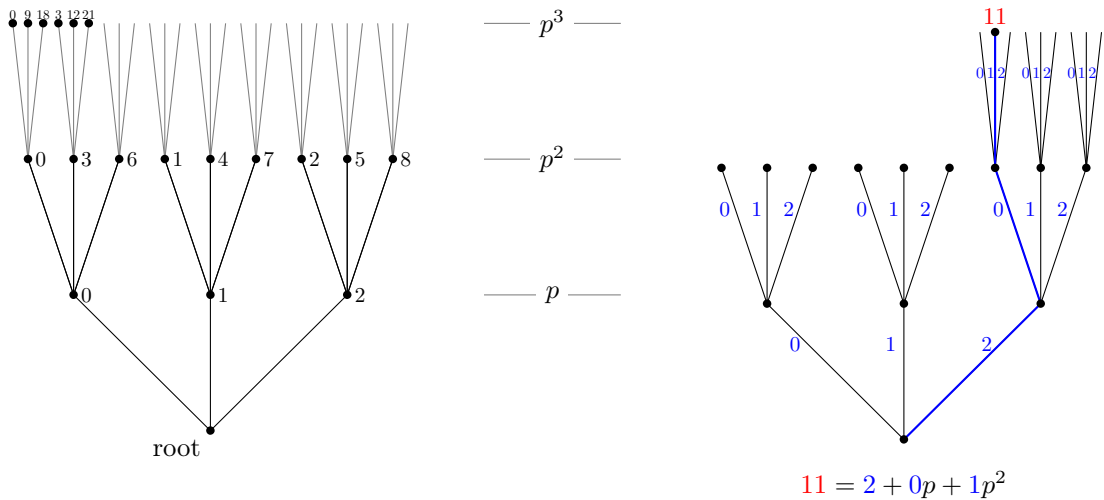

For a given vertex, we can consider that each outgoing edge is indexed by an integer between 0 and $p-1$.
A vertex of the tree corresponds to an integer $x\in\Zz$, the path from the root of the tree to this vertex corresponds to its $p$-adic decomposition, i.e.{} a finite sum $x = \sum_{i\ge0} a_i p^i$ with $0 \le a_i \le p-1$ (for each $i \ge0$).
The infinite paths in the tree correspond to coherent sequences of residues and naturally form a ring, which is exactly the ring of $p$-adic integers $\Zz_p$.

\subsection{Solution tree}

Let $P(X) \in \Zz[X]$ be a polynomial with integer coefficients.
We denote:
\[ 
\Tree(P) = \big\{ (x,e) \in \Omega_p \mid P(x) \equiv 0 \pmod{p^e} \big\} 
\]

This \defi{solution tree} may be finite or infinite.
Let us verify that $\Tree(P)$ is indeed a tree: 
let $(x,e) \vartriangleleft (x',e')$ with $(x',e') \in \Tree(P)$,
then $x \equiv x' \pmod{p^e}$ hence $P(x) \equiv P(x') \equiv 0 \pmod{p^e}$ and thus we also have $(x,e) \in \Tree(P)$.

In particular, infinite paths in $\Tree(P)$ (that is, sequences $(x_e, e)$ with $P(x_e) \equiv 0 \mod p^e$ and $x_{e+1} \equiv x_e \mod p^e$) correspond to roots in $\Zz_p$ of the polynomial $P$, meaning $p$-adic integers $\alpha$ such that $P(\alpha) = 0$ in $\Zz_p$.

\subsection{Thickness}

Let $P \in \Zz[X]$ be a polynomial of degree $d$.
To simplify the presentation throughout this article, we assume that $p$ does not divide $P(X)$ in $\Zz[X]$, in other words, the coefficients of $P$ are not simultaneously all divisible by $p$.
This is not a significant loss of generality; if this assumption were not verified, we would start by writing $P(X) = p^{t_0} Q(X)$ where $p$ does not divide $Q(X)$ and then all the results would apply to $Q(X)$.

\begin{definition}
	\label{def:thickness}
	Let $r \in \Zz$.
	The \defi{thickness} $t$ of $P$ at $r$ is the largest integer such that there exists $Q(X) \in \Zz[X]$ such that:
	\[ 
	P(r+pX) = p^t Q(X)
	\]
	The polynomial $Q$ is the \defi{successor} of $P$ for the root $r$.
\end{definition}
Note that the definition of thickness is only meaningful for the roots of $P$ modulo $p$:
\[ 
P(r) \equiv 0 \pmod{p} \iff t \ge 1 
\]
Also note that by the maximality of $t$, $p$ does not divide $Q(X)$ in $\Zz[X]$.

\subsection{Trunk of a polynomial}

For a polynomial $P \in \Zz[X]$, we now define its \defi{trunk} and the \defi{thicknesses} associated with its vertices.
We will denote the trunk by $\Trunk(P)$. It is a subtree, maybe infinite, of the solution tree $\Tree(P)$. 

We set $P_0 = P$. We define the structure of the trunk inductively, with each level built from the previous one.
At each step, we look for roots modulo $p$ of the current polynomial $P_k$. For each such root $r$, we consider the polynomial $P_k(r + pX)$ and factor out the highest power $p^t$ of $p$ ($t$ is the thickness).
We then define the successor polynomial associated with $r$ as
$P_{k+1}(X) = \frac{1}{p^t} P_k(r + pX)$.

More precisely:
\begin{itemize}
	\item \emph{Level $0$.} 
	We set $(r,k)=(0,0) \in \Trunk(P)$. This is the only vertex of the trunk to which no thickness is associated.
	
	\item \emph{Level $1$.} 
	For each $r_0 \in \llbracket 0, p-1 \rrbracket$ such that $P(r_0) \equiv 0 \pmod{p}$, we compute the decomposition $P(r_0+pX) = p^{t_1} Q(X)$ and include in the trunk the vertex $(r_0,1)$ associated with the thickness $t_1$.
	
	\item \emph{Level $2$.} 
	For the successor $Q$ of $P$ at each $r_0$ from the previous step, we look for solutions $r_1 \in \llbracket 0, p-1 \rrbracket$ such that $Q(r_1) \equiv 0 \pmod{p}$; we compute the decomposition $Q(r_1+pX)=p^{t_2} R(X)$; the pair $(r_0+p r_1, 2)$ is a new element of $\Trunk(P)$ associated with the thickness $t_2$.
	
	\item \emph{From level $k$ to level $k+1$.} 
	By induction, suppose that $(r_0+pr_1+p^2r_2+\cdots+p^{k-1}r_{k-1},k) \in \Trunk(P)$ with the polynomial $P_k$ obtained as a successor of $r_{k-1}$. We look for solutions $r_k \in \llbracket 0, p-1 \rrbracket$ such that $P_k(r_k) \equiv 0 \pmod{p}$; we compute the decomposition $P_k (r_k+pX) = p^{t_{k+1}} P_{k+1}(X)$; the pair $(r_0+pr_1+p^2r_2+\cdots+p^{k}r_{k}, k+1)$ is an element of $\Trunk(P)$ associated with the thickness $t_{k+1}$.
\end{itemize}

The \defi{tree-top function} associates to each vertex $(r,k)$ of $\Trunk(P)$ is the sum of the thicknesses encountered on the path to the root. In other words,
\[
\varphi(r,k) = t_1 + t_2 + \cdots + t_k.
\]
where each $t_i$ is the thickness at level $i$ of the vertex on the path between the root and the vertex $(r,k)$.

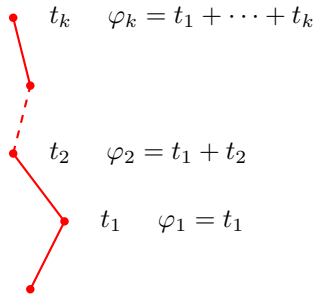
\begin{figure}[H]
	\myfigure{0.6}{
\begin{tikzpicture}[scale=1.5]

\tikzset{
  line/.style = {
  },
  vector/.style = {
    thick,-latex
  },
  dot/.style = {
    insert path={
      node[scale=3]{.}
    }
  }
}

\begin{scope}

 \path
   (0,0) coordinate (P0)
   (0.5,1) coordinate (P1)
   (-0.25,2) coordinate (P2)
   (0,3) coordinate (P3)
   (-0.25,4) coordinate (P4)
 ;

\path 
  (P0) edge[line, thick, red] (P1)
  (P1) edge[line, thick, red] (P2)
  (P2) edge[line, thick, dashed, red] (P3)
  (P3) edge[line, thick, red] (P4)
;


 \path
   (P0) [red, dot] node[below]{}
   (P1) [dot, red] node[black,right=1em]{$t_1$ \quad $\varphi_1 = t_1$}
   (P2) [dot, red] node[black,right=1em]{$t_2$ \quad $\varphi_2 = t_1+t_2$}
   (P3) [dot, red] node[below]{}
   (P4) [dot, red] node[black,right=1em]{$t_k$ \quad $\varphi_k = t_1+\cdots+t_k$}
 ;



\end{scope}

\end{tikzpicture}%

	}
	\caption{Thickness and the tree-top function $\varphi$.}
	\label{fig:toptree}
\end{figure}

\subsection{An example}

Before stating the theorem, let’s go through an example to better understand these concepts. We explain the computation of the trunk illustrated in Figure \ref{fig:ex1}.

\begin{example}
	\label{ex:ex1bis}
	Let $P(X) = (X^2+3)(X^2+3X+9)$ and $p=3$.
	The reduction modulo $p$ of $P$ is $\reduc{P}(X)= X^4$.
	Thus for $r_0 = 0$, we have $P(r_0) \equiv 0 \pmod{3}$.
	The decomposition of $P(r_0+pX)$ is $P(3X) = 3^3 (3X^2+1)(X^2+X+1)$.
	Thus the thickness associated with $r_0$ is $t_1=3$ and the successor of $P$ at $r_0$ is $P_1(X) = (3X^2+1)(X^2+X+1)$.
	The first vertex of the trunk is thus $(r_0,1)=(0,1)$ associated with a thickness $t_1=3$.
	
	We start again, from $P_1$: $\reduc{P_1}(X) = X^2+X+1$ vanishes modulo $3$ at $r_1=1$,
	and the decomposition of $P_1(r_1+pX)$ is $P_1(1+3X) = 3^1 (27X^2+18X+4)(3X^2+3X+1)$.
	Thus the second vertex of the trunk is $(r_0+pr_1,2)=(3,2)$ associated with a thickness $t_2=1$.
	
	The successor of $P_1$ at $r_1$ is $P_2(X) = (27X^2+18X+4)(3X^2+3X+1)$, which does not vanish modulo $p=3$. Thus, the calculations stop here.
	
	In summary, besides the root $(0,0)$, the trunk is composed of vertices $(0,1)$ (with $t_1=3$) and $(3,2)$ (with $t_2=1$). 
\end{example}

\section{Main theorem}
\label{sec:main}

\subsection{From the trunk to the tree}

We recall that:
\[
P(x) \equiv 0 \pmod{p^e} \iff (x,e) \in \Tree(P)
\]
The following theorem indicates how $\Trunk(P)$ determines the roots $\Tree(P)$ of a polynomial $P$, via the tree-top function $\varphi$ associated with the trunk.
What's the benefit? The trunk is easily computed from $P$ and its number of vertices for a fixed level $e$ is bounded by the degree of $P$ (see Section \ref{sec:struct}), unlike the tree $\Tree(P)$, which can have a number of vertices that grows exponentially with level $e$.

\begin{theorem}
\label{th:trunktotree}
\[
P(x) \equiv 0 \pmod{p^e} \iff  \text{ there exists }  (r,k) \in \Trunk(P) \text{ such that } 
\left\lbrace\begin{array}{l}
x \equiv r \pmod{p^k}  \\
\text{and } \varphi(r,k) \ge e
\end{array}\right.
\]
\end{theorem}
Thus, to know if $x$ is a root of $P$ modulo $p^e$, it suffices to check a combinatorial condition on $x$ (modulo a certain $p^k$). 
Since these solutions can be numerous, one might want to simply calculate their number without explicitly listing them all; this will be done in Section \ref{sec:count}.

\medskip

Let us reformulate these results to explain how the solution tree is recovered from the trunk by adding fans.
The \defi{fan} of a vertex $(r,k)$ up to level $h$ is the set of vertices of $\Omega_p$, issued from vertex $(r,k)$ up to level $h$:
\[
\Fan_{\le h}(r,k) = \big\{ (x,l) \in \Omega_p \mid  (r,k) \vartriangleleft (x,l) \text{ and } l \le h \big\}
\]

\begin{figure}[H]
\myfigure{0.6}{
\begin{tikzpicture}[scale=2]

\tikzset{
  line/.style = {
  },
  vector/.style = {
    thick,-latex
  },
  dot/.style = {
    insert path={
      node[scale=3]{.}
    }
  }
}

\def\p{3}
 \path
   (0,0) coordinate (O)
   (-0.5,-1) coordinate (Q)
 ;

\foreach \i in {0,...,2} {
     \path (\i-1,1) coordinate (P\i1);
  }

\foreach \i in {0,...,8} {
     \path ({(\i-4)/3},2) coordinate (P\i2);
  }

\path (Q) edge[line, thick, red, dashed] (O);
\path (Q) edge[line, thick, red, dashed] (-0.75,-0.25);
\path (Q) edge[line, thick, red, dashed] (-0.5,-1.5);
  \foreach \i in {0,1,2} {
     \path (O) edge[line] (P\i1);
  }
  \foreach \i in {0,1,2} {
     \foreach \k in {0,1,2} {
         \pgfmathtruncatemacro\ii{\i+3*\k}
          \path (P\k1) edge[line] (P\ii2);
     }
  }
  \foreach \i in {0,1,2} {
     \foreach \k in {0,1,2} {
         \pgfmathtruncatemacro\ii{\i+3*\k}
          \path (P\k1) edge[line] (P\ii2);
     }
  }

 \path
   (O) [red,dot] node[below right]{$(r,k)$}
   (Q) [dot,red] {}
 ;
  \foreach \i in {0,...,2} {
     \path  (P\i1) [dot] {};
  }
  \foreach \i in {0,...,8} {
     \path  (P\i2) [dot] {};
  }

  \draw[line,thin,gray] (1.75,2) -- ++(2,0) node[midway, black, fill=white, right]{level $h$};
  \draw[line,thin,gray] (1.75,0) -- ++(2,0) node[midway, black, fill=white, right]{level $k$};

  \draw[<->,>=latex,line width=3pt,blue!10] (2,0) -- (2,2) node[midway,right, black] {Fan};
\node[below right, red] at ($(Q)!0.2!(O)$) {Trunk};
\end{tikzpicture}%

}
\caption{A fan.}
\label{fig:fan}
\end{figure}

Thus, Theorem \ref{th:trunktotree} is reformulated as:
\[
\Tree(P) = \bigcup_{(r,k) \in \Trunk(P)} \Fan_{\le  \varphi(r,k)}(r,k)
\]
Since the thickness is always at least $1$, we have $k \leq \varphi(r,k)$, and therefore the vertex $(r,k)$ is indeed an element of $\Fan_{\leq \varphi(r,k)}(r,k)$.
Moreover, if the trunk is a finite tree, then there exists $e\ge0$ such that the equation $P(x)\equiv 0 \pmod{p^e}$ has no integer solutions.

\subsection{First example for the main theorem}

Let us compute the tree of Example \ref{ex:ex1bis}, which is already depicted in Figure \ref{fig:ex0}.

\begin{example}
\label{ex:ex1}
Let $P(X) = (X^2+3)(X^2+3X+9)$ and $p=3$. In the figure below on the left, we have the trunk of $P$, which is a tree with only $3$ vertices (the computation has been done in Example \ref{ex:ex1bis}). The vertex $(0,1)$ has a thickness $t_1=3$ (and thus a tree-top function value of $\varphi_1=t_1=3$); the vertex $(3,2)$ has a thickness $t_2=1$ and thus a tree-top function value of $\varphi_2=t_1+t_2=4$. To obtain the solution tree, below on the right: we start from the trunk of $P$; to the vertex $(0,1)$ we adjoin the fan originating from this vertex that goes up to level $\varphi_1 = 3$; to the vertex $(3,2)$ we adjoin the fan originating from this vertex that goes up to level $\varphi_2 = 4$.

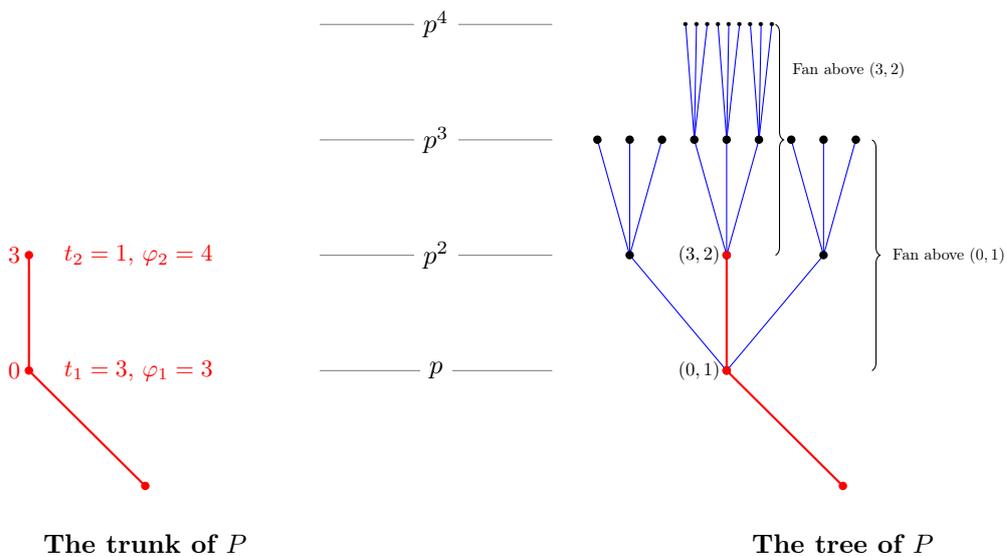
\begin{figure}[H]
\myfigure{0.9}{
\begin{tikzpicture}[scale=1.7]
\usetikzlibrary{decorations.pathreplacing}
\tikzset{
  line/.style = {
  },
  vector/.style = {
    thick,-latex
  },
  dot/.style = {
    insert path={
      node[scale=3]{.}
    }
  },
  smalldot/.style = {
    insert path={
      node[scale=1.5]{.}
    }
  }
}

\begin{scope}
 \path
   (0,0) coordinate (O)
 ;

\foreach \i in {0,...,2} {
     \path (\i-1,1) coordinate (P\i1);
  }

\foreach \i in {0,...,8} {
     \path ({(\i-4)/3},2) coordinate (P\i2);
  }
\foreach \i in {0,...,26} {
     \path ({(\i-13)/9},3) coordinate (P\i3);
  }
\foreach \i in {0,...,80} {
     \path ({(\i-40)/27},4) coordinate (P\i4);
  }

\path (O) edge[line,  thick, red] (P01);
 

\path (P01) edge[line,  thick, red] (P12);


 \path
   (O) [red, dot, red] node[below]{}
   (P01) [dot, red] node[left,scale=0.9]{$0$} node[right=1em,scale=0.9]{$t_1 = 3$, $\varphi_1 = 3$}
   (P12) [dot, red] node[left,scale=0.9]{$3$} node[right=1em,scale=0.9]{$t_2 = 1$, $\varphi_2 = 4$}
 ;

\node at (0,-0.5) {\bf The trunk of $P$};

  \draw[line,thin,gray] (1.5,1) -- ++(2,0) node[midway, black, fill=white,]{$p$};
  \draw[line,thin,gray] (1.5,2) -- ++(2,0) node[midway, black, fill=white, ]{$p^2$};
  \draw[line,thin,gray] (1.5,3) -- ++(2,0) node[midway, black, fill=white, ]{$p^3$};
  \draw[line,thin,gray] (1.5,4) -- ++(2,0) node[midway, black, fill=white, ]{$p^4$};

\end{scope}

\begin{scope}[xshift=6cm]
 \path
   (0,0) coordinate (O)
 ;

\foreach \i in {0,...,2} {
     \path (1*\i-1,1) coordinate (P\i1);
  }

\foreach \i in {0,...,8} {
     \path ({(2.5*\i-5.5)/3},2) coordinate (P\i2);
  }
\foreach \i in {0,...,26} {
     \path ({(2.5*\i-19)/9},3) coordinate (P\i3);
  }
\foreach \i in {0,...,80} {
     \path ({(2.5*\i-59)/27},4) coordinate (P\i4);
  }

\path (O) edge[line, thick, red] (P01);
 
\path (P01) edge[line, blue] (P02);
\path (P01) edge[line, blue] (P22);

\path (P01) edge[line, thick, red] (P12);

\path (P02) edge[line, blue] (P03);
\path (P02) edge[line, blue] (P13);
\path (P02) edge[line, blue] (P23);

\path (P12) edge[line, blue] (P33);
\path (P12) edge[line, blue] (P43);
\path (P12) edge[line, blue] (P53);

\path (P22) edge[line, blue] (P63);
\path (P22) edge[line, blue] (P73);
\path (P22) edge[line, blue] (P83);

\path (P33) edge[line, thin, blue] (P94);
\path (P33) edge[line, thin, blue] (P104);
\path (P33) edge[line, thin, blue] (P114);

\path (P43) edge[line, thin, blue] (P124);
\path (P43) edge[line, thin, blue] (P134);
\path (P43) edge[line, thin, blue] (P144);

\path (P53) edge[line, thin, blue] (P154);
\path (P53) edge[line, thin, blue] (P164);
\path (P53) edge[line, thin, blue] (P174);


  \foreach \i in {0,...,0} {
     \path  (P\i1) [dot] {};
  }
  \foreach \i in {0,...,2} {
     \path  (P\i2) [dot] {};
  }
  \foreach \i in {0,...,8} {
     \path  (P\i3) [dot] {};
  }
  \foreach \i in {9,...,17} {
     \path  (P\i4) [smalldot] {};
  }
 \path
   (O) [red, dot, red] {}
   (P01) [dot, red] {}
   (P12) [dot, red] {} 
 ;

 \path
   (P01)  node[left,scale=0.7]{$(0,1)$}
   (P12) node[left,scale=0.7]{$(3,2)$}
 ;

\draw [decorate,decoration={brace,amplitude=3pt,mirror},xshift=0.5cm,yshift=0pt]
      (-0.25,1) -- ++(0,2) node [midway,right,xshift=.2cm,scale=0.6] {Fan above $(0,1)$};
\draw [decorate,decoration={brace,amplitude=3pt,mirror},xshift=0.5cm,yshift=0pt]
      (-1.08,2) -- ++(0,2) node [pos=0.8,right,xshift=.15cm,scale=0.6] {Fan above $(3,2)$};

\node at (0,-0.5) {\bf The tree of $P$};

\end{scope}

\end{tikzpicture}%

}
\caption{The tree from the trunk.}
\label{fig:ex1}
\end{figure}

From the tree, we read the solutions of the equation $P(x)\equiv 0 \pmod{p^e}$ for different values of $e$. The formula from Corollary \ref{cor:count} will allow us to count the number $N_e$ of solutions without explicitly enumerating them.  
\[
\begin{array}{ccc}
	p^e & \text{ solutions } & N_e \\ \hline
	3^1 & 0 & 1 \\
	3^2 & 0,3,6 & 3 \\
	3^3 & 0, 3, 6, 9, 12, 15, 18, 21, 24 & 9\\
	3^4 & 3, 12, 21, 30, 39, 48, 57, 66, 75 & 9\\
\end{array}
\]

For $e \ge 5$, the equation has no solutions.
\end{example}

\subsection{Second example}

We will consider an example in which the congruence $P(x)\equiv0\pmod{p^e}$ has solutions for every $e \ge 1$. One situation where this happens is when there is a simple root modulo $p$, that is $P(x_1)\equiv0\pmod p$ but $P'(x_1)\not\equiv0\pmod p$. Then Hensel’s lemma (see \cite{Apo1976} or \cite{NZM}) shows that this root can be lifted indefinitely to solutions modulo $p^2,p^3,\dots$, thereby producing an infinite branch in the solution tree whose vertices all have thickness $1$ (see Lemma \ref{lem:taylor}).

\begin{theorem}[Hensel's Lemma]
	Let $P(x)\in\mathbb{Z}[x]$ and $x_1\in\mathbb{Z}$ be such that
	$P(x_1)\equiv0\pmod{p}$ and $P'(x_1)\not\equiv0\pmod{p}$.
	Then, for every integer $e\ge1$, there exists a unique integer $x_e$ (determined modulo $p^e$) satisfying	
	$P(x_e)\equiv0\pmod{p^e}$ and $x_e\equiv x_1\pmod{p}$.
\end{theorem}

The idea of the proof is a variant of Newton’s method for finding roots. 
We proceed by induction on the exponent $e$.  The first step is a Taylor expansion around the known root $x_1$:
$$
P(x_1 + h p) \equiv P(x_1)+h p P'(x_1)\pmod{p^2}.
$$

Since $p\mid P(x_1)$, write $P(x_1)=p\cdot C$ and let $D$ be an inverse of $P'(x_1)$ modulo $p$.  Then set
$h_0  = -\frac{P(x_1)}{p} D = -CD$.
By construction,
$$
P\big(x_1 + h_0 p\big)
\equiv P(x_1) + h_0 p P'(x_1)
\equiv pC\big(1 - D P'(x_1)\big)
\equiv0\pmod{p^2},
$$
so $x_2 = x_1 + h_0 p$ is indeed a root modulo $p^2$. From there one continues inductively to lift to all higher powers of $p$.

\begin{example}
\label{ex:ex2}
Let $P(X)= X(X-1)^2+5^2$ and consider $p=5$.

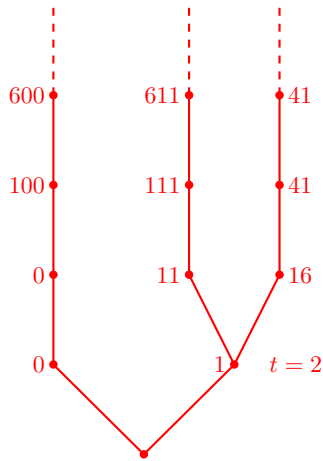
\begin{figure}[H]
	\myfigure{0.7}{
\begin{tikzpicture}[scale=1.7]

\tikzset{
  line/.style = {
  },
  vector/.style = {
    thick,-latex
  },
  dot/.style = {
    insert path={
      node[scale=3]{.}
    }
  },
  smalldot/.style = {
    insert path={
      node[scale=1.5]{.}
    }
  }
}

\begin{scope}

 \path
   (0,0) coordinate (P0)
   (-1,1) coordinate (P1)
   (-1,2) coordinate (P2)
   (-1,3) coordinate (P3)
   (-1,4) coordinate (P4)
	(-1,5) coordinate (P5)
;

\def\esp{0.5}
 \path
   (1,1) coordinate (Q1)
   (1-\esp,2) coordinate (Q12)
   (1-\esp,3) coordinate (Q13)
   (1-\esp,4) coordinate (Q14)
	(1-\esp,5) coordinate (Q15)
   (1+\esp,2) coordinate (Q22)
   (1+\esp,3) coordinate (Q23)
   (1+\esp,4) coordinate (Q24)
	(1+\esp,5) coordinate (Q25)
;

\path 
(P0) edge[line,  thick, red] (P1)
(P1) edge[line,  thick, red] (P2)
(P2) edge[line,  thick, red] (P3)
(P3) edge[line,  thick, red] (P4)
(P4) edge[line,  thick, red,dashed] (P5)
;
\path 
(P0) edge[line,  thick, red] (Q1)
(Q1) edge[line,  thick, red] (Q12)
(Q12) edge[line,  thick, red] (Q13)
(Q13) edge[line,  thick, red] (Q14)
(Q14) edge[line,  thick, red,dashed] (Q15)
; 
\path 
(Q1) edge[line,  thick, red] (Q22)
(Q22) edge[line,  thick, red] (Q23)
(Q23) edge[line,  thick, red] (Q24)
(Q24) edge[line,  thick, red,dashed] (Q25)
;


 \path
   (P0) [red, dot, red] node[below]{}
   (P1) [dot, red] node[left,scale=0.9]{$0$} 
   (P2) [dot, red] node[left,scale=0.9]{$0$} 
   (P3) [dot, red] node[left,scale=0.9]{$100$} 
   (P4) [dot, red] node[left,scale=0.9]{$600$} 
 ;

\path
   (Q1) [red, dot, red] node[left,scale=0.9]{$1$} node[right=1em,scale=0.9]{$t = 2$}
   (Q12) [dot, red] node[left,scale=0.9]{$11$} 
   (Q13) [dot, red] node[left,scale=0.9]{$111$} 
   (Q14) [dot, red] node[left,scale=0.9]{$611$} 
   (Q22) [dot, red] node[right,scale=0.9]{$16$} 
   (Q23) [dot, red] node[right,scale=0.9]{$41$} 
   (Q24) [dot, red] node[right,scale=0.9]{$41$} 
 ;



\end{scope}

\end{tikzpicture}%

	}
	\caption{Trunk with infinite branches. All non-marked thicknesses equal $1$.}
	\label{fig:ex2}
\end{figure}

The root $0$ is a simple root of $P$ modulo $5$, meaning that  
$P'(0) \not\equiv 0 \pmod{5}$.   
By Hensel’s lemma, this root can be lifted indefinitely to solutions modulo each $p^e$.  
This results in an infinite branch of the trunk (on the left in Figure \ref{fig:ex2}). At each level of this branch, the thickness is $1$, so the part of the solution tree corresponding to this branch coincides with the branch itself. In summary, for each level $e$, there is a unique solution $x$ satisfying $x \equiv 0 \pmod{p}$ and $P(x) \equiv 0 \pmod{p^e}$. 

The root $1$ is not a simple root; it has thickness $2$. After level $1$, the trunk splits into two infinite branches, each of thickness $1$. The solution tree can be recovered from the trunk by Theorem \ref{th:trunktotree} (but is not pictured in Figure \ref{fig:ex2}).

	\[
\begin{array}{cccc}
	p^e & \text{ solution above $0$} & \text{ solutions above $1$} & N_e \\ \hline
	5^1 & 0 & 1 & 2 \\
	5^2 & 0 & 1, 6, 11, 16, 21 & 6 \\
	5^3 & 100 & 11, 16, 36, 41, 61, 66, 86, 91, 111, 116 & 11\\
	5^4 & 600 & 41, 111, 166, 236, 291, 361, 416, 486, 541, 611 & 11\\
\end{array}
\]
For all $e\ge5$, the number of solutions remains $N_e=11$.

\end{example}

\begin{remark*}
There may exist infinite branches with vertices of thickness greater than 1; an example will be given in Section \ref{sec:degtwo}. Such infinite branches correspond to multiple roots of the polynomial $P(X)$ in the ring  $\Zz_p$ of $p$-adic integers, and are therefore associated with multiple factors in the decomposition of $P(X)$ into irreducible factors in $\Zz[X]$. It is also possible to detect whether a branch beginning with vertices of thickness greater than 1 will extend to infinity; see \cite[Section 3]{DS2020}.
\end{remark*}

\section{Thickness}

In this section, we provide further information and properties about the thickness.

\subsection{Characterization by Taylor's formula}

\begin{lemma}
\label{lem:taylor}
\[
t = \min_{i\ge0} \val_p \left( \frac{P^{(i)}(r)}{i!} p^i\right)
\]
And in particular $t \le d$.
\end{lemma}

We recall that $\val_p(x)$, the \defi{valuation} at $p$ of an integer $x$, is the largest exponent $i$ such that $p^i$ divides $x$. For example $\val_p(p^i) = i$. Also, for any $r \in \Zz$ and $i\ge0$, $\frac{P^{(i)}(r)}{i!}$ is an integer.
As an application of this lemma, $t = 1$ iff $r$ is a simple root modulo $p$, i.e.~$P(r) \equiv 0 \pmod{p}$ and $P'(r) \not\equiv 0 \pmod{p}$.

\begin{proof}
The Taylor formula for $P$ around the root $r$ is written:
\[
P(r+X) = P(r) + P'(r)X + \frac{P''(r)}{2!}X^2 + \cdots + \frac{P^{(i)}(r)}{i!} X^i + \cdots + \frac{P^{(d)}(r)}{d!} X^d
\]
This gives:
\[
P(r+pX) = P(r) + P'(r)pX + \frac{P''(r)}{2!} p^2 X^2 + \cdots + \frac{P^{(i)}(r)}{i!} p^i X^i + \cdots 
\]

Let $t$ be the thickness of $P$ at $r$ and $t' = \min_{i\ge0} \val_p \left( \frac{P^{(i)}(r)}{i!} p^i\right)$.

Since $p^t$ divides the polynomial $P(r+pX)$, then $p^t$ divides all the coefficients $\frac{P^{(i)}(r)}{i!} p^i$ of $P(r+pX)$, thus $t \le t'$.
Conversely, by Taylor's formula, $p^{t'}$ divides all the coefficients of $P(r+pX)$, hence $t' \le t$.
\end{proof}

\begin{lemma}
\label{lem:multiplicity}
The thickness $t$ of $P$ at $r$ is less than or equal to the multiplicity $\mult(r)$ of the root $r$ as a root of the polynomial $\reduc{P} \in \Zz/p\Zz[X]$.
\end{lemma}

We will also prove in Lemma \ref{lem:noderule} that the thickness can only decrease as we
ascend the tree.

\begin{proof}
To simplify the proof, and without loss of generality, we can assume $r=0$.
Let us write $P(X)= \sum_{0 \le i \le d} {a_i X^i}$.
Then $p^t$ divides $P(pX)$, so $p^t$ divides the $a_i p^i$ (for $0 \le i \le d$).
Thus $p$ divides $a_i$ for $i < t$, so after reduction modulo $p$,
$\reduc{P}(X) = \reduc{a_t} X^t + \cdots + \reduc{a_d} X^d$ factors through $X^t$.
Thus the multiplicity of $r=0$ as a root of $\reduc P$ is greater than or equal to $t$.
\end{proof}


\subsection{Residual degree}

\begin{definition}
	Let $P\in\Zz[X]$ of thickness $t$ at $r\in\Zz$, associated with the decomposition $P(r+pX) = p^t Q(X)$.
	The \defi{residual degree} of $P$ at $r$, denoted $s$, is the degree of the reduction of $Q$ in $\Zz/p\Zz[X]$. In other words, $s = \deg{\reduc{Q}}$.
\end{definition}

\begin{lemma}
	\label{lem:output}
	The residual degree $s$ is at most $t$, and is the largest integer $i\ge0$ such that
	\[
	\val_p \left( \frac{P^{(i)}(r)}{i!} p^i \right) = t.
	\]
\end{lemma}

\begin{example}
	\label{ex:output}
	Let $P(X)= X^3 + pX^2 + pX$. The thickness of the root $r=0$ is $t=2$ because $P(pX) = p^2 Q(X)$ where $Q(X) = pX^3+pX^2+X$. The residual degree is $s=1$ because the reduction of $Q$ modulo $p$ is of degree $1$.
\end{example}

\begin{proof}
	To simplify the writing of the proof, we can again assume without loss of generality that $r=0$ and write $P(pX) = p^t Q(X)$. Let $P(X) = \sum_{0 \le i \le d} a_i X^i$. By Taylor's formula, $a_i = \frac{P^{(i)}(r)}{i!}$. By hypothesis $p^t$ divides $P(pX)$, so $p^t \divides a_i p^i$ for all $0 \le i \le d$.
	Thus:
	\begin{align*}
		\deg \reduc{Q} = s
		&\iff p^{t+1} \notdivides a_s p^s \text{ and } p^{t+1} \divides a_i p^i \text{ for all } i>s \\
		&\iff \val_p(a_s p^s) = t \text{ and } \val_p(a_i p^i) > t \text{ for all } i>s \\
		&\iff s \text{ is the largest integer such that } \val_p(a_i p^i) = t 
	\end{align*}
	Finally $t \ge s$ because:
	\[
	t 
	= \min_{0 \le i \le d} \left[ \val_p \left( \frac{P^{(i)}(r)}{i!} \right) + i \right]
	= \val_p \left( \frac{P^{(s)}(r)}{s!} \right) + s
	\]
	
\end{proof}

\subsection{Node rule}

\begin{lemma}[Node rule]
	\label{lem:noderule}
	Let $(r,k) \in \Trunk(P)$ with thickness $t$ and residual degree $s$.
	Let $(r_i,k+1) \in \Trunk(P)$ of thickness $t_i$ be the children of $(r,k)$, $i=1,\ldots,l$.
	Then:
	\[ 
	t_1+t_2+\cdots+t_l \le s \le t
	\]
\end{lemma}	

\begin{figure}[H]
	\myfigure{0.6}{
\begin{tikzpicture}[scale=2]

\tikzset{
  line/.style = {
  },
  vector/.style = {
    thick,-latex
  },
  dot/.style = {
    insert path={
      node[scale=3]{.}
    }
  }
}

\begin{scope}

 \path
   (0,0) coordinate (P0)
   (-1,1) coordinate (P1)
   (-0.5,1) coordinate (P2)
   (0,1) coordinate (P3)
   (0.5,1) coordinate (P4)
 ;

\path 
  (P0) edge[line,  thick, red] (0,-0.5)
  (P0) edge[line,  thick, red] (P1)
  (P0) edge[line,  thick, red] (P2)
  (P0) edge[line,  thick, red] (P4)
;


 \path
   (P0) [dot, red] node[black,below right]{$t$}
   (P1) [dot, red] node[black,above]{$t_1$}
   (P2) [dot, red] node[black,above]{$t_2$}
   (P4) [dot, red] node[black,above]{$t_l$}
 ;

\node[scale=1.5] at (0.05,0.8) {$\cdots$};


\end{scope}

\end{tikzpicture}%

	} 
	\caption{The node rule: $t_1+t_2+\cdots+t_l \le t$.}
	\label{fig:noderule}
\end{figure}
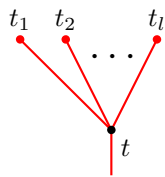

\begin{proof}
	Let $P(r+pX) = p^t Q(X)$.
	Let $r_i = r + \rho_i p^k$, $i=1,\ldots,l$ where $\rho_i$ are the roots of $Q$ modulo $p$.
	Consider the decomposition of the reduction of $Q$ modulo $p$ according to its roots:
	\[
	\reduc{Q}(X) = (X-\rho_1)^{\mu_1}\cdots (X-\rho_l)^{\mu_l} I(X) \in \Zz/p\Zz[X]
	\]
	where $I(X)$ has no roots modulo $p$.
	By Lemma \ref{lem:multiplicity}, the thickness is less than or equal to the multiplicity: $t_i \le \mu_i$.
	Thus, remembering that $\deg \reduc{Q}$ is by definition the residual degree $s$ and that $s \le t$ (Lemma \ref{lem:output}):
	\[
	\sum_{1 \le i \le l} t_i 
	\le \sum_{1 \le i \le l} \mu_i 
	\le \deg \reduc{Q}
	= s 
	\le t
	\]
\end{proof}

\section{Construction of the solution tree from the trunk}

Now that we have defined the trunk and explained its main properties, it is time to prove Theorem \ref{th:trunktotree} which explains how to compute the tree $\Tree(P)$ of solutions $P(x) \equiv 0 \pmod{p^e}$ from the trunk $\Trunk(P)$, via the formula:
\[
P(x) \equiv 0 \pmod{p^e} \iff  \text{ there exists }  (r,k) \in \Trunk(P) \text{ such that } 
x \equiv r \pmod{p^k} \text{ and } \varphi(r,k) \ge e
\]

\subsection{Tree-top function}

For $(r,k) \in \Trunk(P)$, let $t_1,\ldots,t_k$ be the thicknesses associated with the path from the root to the vertex $(r,k)$. Let $\varphi = \varphi(r,k) = t_1 + \cdots +t_k$ be the value of the tree-top function at this vertex.

\begin{lemma}
	\label{lem:toptree}
There exists a decomposition
	\[
	P(r+p^k X) = p^{\varphi} Q(X)	
	\]
	where $Q(X) \in \Zz[X]$. 
\end{lemma}

\begin{proof}
Let $r= r_0+r_1p+r_2p^2+\cdots+r_{k-1}p^{k-1}$ be the $p$-adic expansion of $r$.
Then
$P(r_0+pX) = p^{t_1} P_1(X)$ where $P_1$ denotes the successor of $P$ for the root $r_0$,
hence 
\[
P(r_0+r_1p+p^2X) = P\big(r_0+p (r_1+pX)\big) = p^{t_1} P_1(r_1+pX) = p^{t_1+t_2}P_2(X)
\]
where $P_2$ is the successor of $P_1$ for the root $r_1$.
By induction, $P(r+p^k X) = p^{t_1+\cdots+t_k}  P_{k}(X)$ with $P_{k}(X) \in \Zz[X]$.
\end{proof}

\subsection{Proof of Theorem \ref{th:trunktotree}}

Let $(x,e) \in \Omega_p$. We want to know if $P(x) \equiv 0 \pmod{p^e}$, that is if $(x,e) \in \Tree(P)$.
Let $(r,k) \in \Trunk(P)$ be the most recent ancestor of $(x,e)$ belonging to the trunk of $P$. We have $x = r + p^{k} y$ (where $y\in\Zz$).

\begin{figure}[H]
	\myfigure{0.7}{
\begin{tikzpicture}[scale=2]

\tikzset{
  line/.style = {
  },
  vector/.style = {
    thick,-latex
  },
  dot/.style = {
    insert path={
      node[scale=3]{.}
    }
  }
}

\def\p{3}
 \path
   (0,0) coordinate (O)
   (0.5,-1) coordinate (Q)
 ;

\foreach \i in {0,...,2} {
     \path (\i-1,1) coordinate (P\i1);
  }

\foreach \i in {0,...,8} {
     \path ({(\i-4)/3},2) coordinate (P\i2);
  }


  \foreach \i in {0,1,2} {
     \path (O) edge[line,gray] (P\i1);
  }

\path (Q) edge[line, thick, red] (O);
\path (O) edge[line, thick, red] (P21);
\path (Q) edge[line, thick, dashed, red] (1,-0.25);
\path (Q) edge[line, thick, dashed, red] (0.5,-1.25);

  \foreach \i in {0,1,2} {
     \foreach \k in {0,1,2} {
         \pgfmathtruncatemacro\ii{\i+3*\k}
          \path (P\k1) edge[line,gray] (P\ii2);
     }
  }
  \foreach \i in {0,1,2} {
     \foreach \k in {0,1,2} {
         \pgfmathtruncatemacro\ii{\i+3*\k}
          \path (P\k1) edge[line,gray] (P\ii2);
     }
  }

\path (O) edge[line,  thick, blue] (P01); 
\path (P01) edge[line,  thick, blue] (P22); 


  \foreach \i in {0,...,2} {
     \path  (P\i1) [gray,dot] {};
  }
  \foreach \i in {0,...,8} {
     \path  (P\i2) [gray,dot] {};
  }
 \path
   (O) [red,dot] node[below left,scale=0.9]{$(r,k)$}
   (Q) [dot,red] {}
   (P22) [blue,dot] node[above,scale=0.9]{$(x,e)$}
   (P01) [blue,dot] {}
 ;
\node[below left, red,scale=0.9] at ($(Q)!0.2!(O)$) {Trunk};



\end{tikzpicture}%

	} 
	\caption{The most recent ancestor of $(x,e)$.}
	\label{fig:ancester}
\end{figure}
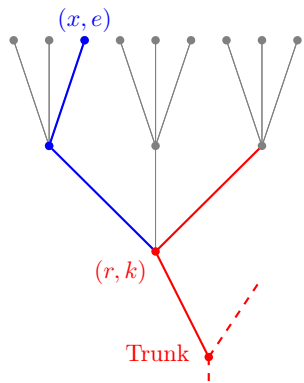

Using the notations of Lemma \ref{lem:toptree}, we have:
\[
P(r + p^{k}X) = p^{\varphi(r,k)} Q(X)
\]
where $\varphi(r,k)$ is the value of the tree-top function.
But additionally, we know that $Q(y) \not\equiv 0 \pmod{p}$ because otherwise $(r,k)$ would not be the most recent ancestor of $(x,e)$ belonging to the trunk.

Thus:
\begin{align*}
	P(x) \equiv 0 \pmod{p^e}
	&\iff P(r+p^{k}y) \equiv 0 \pmod{p^e} \\
	&\iff p^{\varphi(r,k)} Q(y) \equiv 0 \pmod{p^e} \\  	
	&\iff  \varphi(r,k) \ge e
\end{align*}
where for the last equivalence we used that $Q(y)\not\equiv 0 \pmod{p}$.

Moreover, for a given solution $x$, such a pair $(r,k)$ is unique if we impose the condition:  
\[
\varphi(r,k) -t_k < e \le \varphi(r,k)
\]
where $t_k$ is the thickness of the vertex $(r,k)$.  
This means that a vertex $(r,k)$ of the trunk corresponds uniquely to roots whose level is strictly greater than $\varphi(r,k) -t_k$ but less than or equal to $\varphi(r,k)$.


\section{Number of solutions}
\label{sec:count}

We will extract from Theorem \ref{th:trunktotree} a formula that allows us to directly compute the number of solutions from the trunk.

\subsection{Formula}
\label{ssec:count}

\begin{corollary}
	\label{cor:count}
	The number of solutions of the equation $P(x) \equiv 0 \pmod{p^e}$ in $\Zz/ p^e \Zz$ is:
	\[
	N_e = \sum_{\substack{(r,k) \in \Trunk(P) \\ \varphi(r,k) -t_k \, < \, e \, \le \, \varphi(r,k)}} 
	p^{e-k}.	
	\]
\end{corollary}

\begin{proof}
We have seen that the fan originating from the vertex $(r,k)$ of the trunk produces solutions up to the height $\varphi(r,k)$.  
Let $(r^-,k-1)$ be the direct predecessor of $(r,k)$ (that is, the vertex of the trunk adjacent to $(r,k)$ on the side of the root). By denoting $t$ as the thickness of $(r,k)$ and $\varphi(r,k)$ as the tree-top function, we obtain: $\varphi(r^-,k-1) = \varphi(r,k) - t$.
Thus, the vertex $(r^-,k-1)$ produces solutions up to the height $\varphi(r,k) - t$.  
Therefore the solutions $x$, whose height $e$ satisfies $\varphi(r,k) - t < e \le \varphi(r,k)$,
are uniquely associated with the single element $(r,k)$ of the trunk.

How many solutions does such a vertex $(r,k)$ of the trunk produce?  
The fan originating from $(r,k)$ has:  
$1$ vertex at level $k$,  
$p$ vertices at level $k+1$,
$p^2$ vertices at level $k+2$, and so on.  
Thus, for a given level $e$, we associate $p^{e-k}$ solutions. The condition $\varphi(r,k) - t < e \leq \varphi(r,k)$ ensures that this level $e$ is reached by this vertex $(r,k)$ but not by any other vertices of the trunk.
\end{proof}

\subsection{Example}
\label{ssec:count-example}

The formula from Corollary \ref{cor:count} allows us to count the number $N_e$ of solutions without explicitly enumerating them. 
We return to our favorite example, whose trunk was computed in Example \ref{ex:ex1bis} and whose solutions were computed in Example \ref{ex:ex1}.

\begin{example}
	Let $P(X) = (X^2+3)(X^2+3X+9)$ and $p=3$.  
	The vertex $(0,1)$ at level $k=1$ of the trunk satisfies $\varphi_1=3$ and $t_1=3$. It contributes to the solutions modulo $p^e$ for levels $e$ satisfying  
	$\varphi_1 - t_1 < e \le \varphi_1$, that is, $e = 1,2,3$. Thus, $N_e = p^{e-k} = p^{e-1} = 3^{e-1}$ for $e = 1,2,3$.
	
	The vertex $(3,2)$ at level $k=2$ of the trunk satisfies $\varphi_2=4$ and $t_2=1$. It contributes to the solutions modulo $p^e$ for levels $e$ satisfying $\varphi_2 - t_2 < e \le \varphi_2$, which means only at level $e=4$, and thus $N_4 = p^{4-2} = 9$.
	For $e \geq 5$, we have $N_e = 0$.
	
\end{example}

\section{Structure of solutions}
\label{sec:struct}

\subsection{Algorithmic aspects}

This is a well-studied aspect (see references in Section \ref{ssec:ref}).
Here we will just explain that expressing the set of solutions in a compact form is easy, meaning it is done in polynomial complexity (depending on the degree of the polynomial and the level $e$; the prime $p$ being fixed), even though the number of these solutions can exponentially depend on these data.

The main reason this is possible is that the trunk is considerably simpler than the tree, even though by Theorem \ref{th:trunktotree} they are combinatorially equivalent data. Indeed, the number of vertices of $\Trunk(P)$
located at a fixed level is bounded by $d$ (the degree of $P$).
The proof for level $k=1$ is simply that the number of roots of
$P$ on the field $\Zz/p\Zz$ is less than $d$. This remains true for any level by induction thanks to the node rule (Lemma \ref{lem:noderule}).

\medskip

Let’s outline the algorithm for computing the trunk and the set of solutions of the equation $P(x) \equiv 0 \pmod{p^e}$.

\textbf{Data.} $p$ a prime number; $P \in \Zz[X]$ of degree $d$; $e \ge 0$ an integer.

\textbf{Goal.} Compute all solutions $P(x) \equiv 0 \pmod{p^e}$ for $x \in \Zz/p^e\Zz$.

\textbf{Step a.} Find solutions modulo $p$: solve $P(x) \equiv 0 \pmod{p}$ by exhaustive search on $0 \le x \le p-1$. The polynomial $P$ of degree $d$ on the field $\Zz/p\Zz$ has at most $d$ roots (and at most $p$ distinct roots). Each solution gives a vertex of the trunk.

\textbf{Step b.} Compute the thickness and decomposition $P(r+pX)=p^t Q(X)$ for each root $r$ from the previous step. This is done through a sequence of elementary operations:
(i) translation $P(X) \to P(a+X)$; (ii) substitution $P(X) \to P(pX)$; (iii) coefficient valuation. 
Associate the thickness with the corresponding vertex of the trunk.

\textbf{Iteration.} Each successor of $Q$ from the previous step, once reduced modulo $p$, has a degree equal to the residual degree (thus less than or equal to $d$, see Lemma \ref{lem:output}), and additionally, by the node rule (Lemma \ref{lem:noderule}), the total number of vertices for a given level is bounded by $d$. Thus each step \textbf{a} or \textbf{b} is repeated at most $d \cdot e$ times.

\medskip

The algorithms for calculating the trunk, the solution tree, and the formula for the number of solutions have been implemented via the computer algebra system \emph{Sage} \cite{sage}.

\subsection{Degrees}

Let $P \in \Zz[X]$.
\begin{itemize}
	\item We denote $d$ as the degree of $P$ in $\Zz[X]$.
	\item We denote $d_p$ as the degree in $\Zz/p\Zz[X]$ of the reduction of $P$ modulo $p$.
	\item We denote $d_{\Trunk}$ as the number of leaves of the trunk $\Trunk(P)$; each infinite branch of the tree counts as a leaf, in addition to the finite leaves.
\end{itemize}

These quantities allow for a rough estimate of the complexity of the trunk of $P$.
\begin{lemma}
	\label{lem:dtrunk}	
	\[
	d_{\Trunk} \le d_p \le d
	\]
\end{lemma}	

\begin{proof}
	The second inequality is obvious. Let's justify the first.
	By the node rule, Lemma \ref{lem:noderule}, we prove by induction on $K \ge 1$ that $d_p \ge \sum_{(r,k) \in X(\Trunk_{\le K}(P))} t(r,k)$,
	where $X(\Trunk_{\le K}(P))$ denotes the set of leaves of the trunk $\Trunk(P)$ truncated at level $K$.
	For $(r,k) \in \Trunk(P)$, $t(r,k) \ge 1$, so $d_p \ge d_{\Trunk}$.
\end{proof}

\subsection{Solutions}

Theorem \ref{th:trunktotree} provides a combinatorial characterization of the solutions of equation $P(x) \equiv 0 \pmod{p^e}$. Now, we will provide a more arithmetic description of these solutions.

For $x\in\Zz$, let $|x|_p = p^{-\val_p(x)}$ denote the $p$-adic absolute value
and $B(r,p^{-k})$ the associated closed ball:
$B(r,p^{-k}) = \{ x \in \Zz \mid \exists n \in \Zz, x = r + np^k \}$.
In other words, $x \in B(r,p^{-k})$ if and only if $p^k$ divides $x-r$.

The set of descendants of $(r,k)$ in the $p$-adic congruence tree $\Omega$, which is an infinite fan stemming from $(r,k)$, is thus also the set of $(x,e)$ where $x \in B(r,p^{-k})$ (and $e \ge k$).
We will consider $B(r,p^{-k}) \cap \llbracket0,p^e-1\rrbracket$ which is the intersection of the fan issued from $(r,k)$ with the level $p^e$.

Fix $e \ge 1$, we denote by $\mathcal{S}_e$, the set of solutions $x$, with $0 \le x \le p^e-1$, of the equation $P(x) \equiv 0 \pmod{p^e}$.
In other words, $\mathcal{S}_e$ corresponds exactly to the set of vertices of the $\Tree(P)$ having exactly level $e$.

\begin{proposition}
	\label{prop:disjoint}
	The set of solutions $\mathcal{S}_e$ is the union
	of at most $d_{\Trunk}$ disjoint subsets $B(r_i,p^{-k_i}) \cap \llbracket0,p^e-1\rrbracket$.
\end{proposition}	

This proposition appears in \cite[Proposition 1]{BLQ2013} and \cite[Proposition 3]{DS2020}. 
\begin{proof}
	According to Theorem \ref{th:trunktotree}:	
	\[
	\mathcal{S}_e
	= \bigcup_{\substack{(r,k) \in \Trunk(P) \\ \varphi(r,k) - t_k < e \le \varphi(r,k)}} B(r,p^{-k}) \cap \llbracket0,p^e-1\rrbracket
	\]
	In particular all the solutions $x$ in a ball $B(r,p^{-k}) \cap \llbracket0,p^e-1\rrbracket$ are associated with the same element $(r,k)$ of the trunk.

	The discussion in Section \ref{ssec:count} proves that this element $(r,k)$ is unique, i.e.{} the balls
$B(r_i,p^{-k_i}) \cap \llbracket0,p^e-1\rrbracket$ are disjoint.
In other words, for a path of the trunk from the root to a leaf (possibly in the form of an infinite branch) there is at most one element $(r,k)$ in the former decomposition of $\mathcal{S}_e$. It implies the bound on the number of balls. 
\end{proof}

\section{Case of degree two polynomials}
\label{sec:degtwo}

Let $p>2$ be a prime number.
Consider the case of a polynomial of degree $2$:
\[
P(X) = aX^2+bX+c \in \Zz[X]
\]
with $p \notdivides a$.
All the configurations begin above the root with a base consisting of a stem of $\ell$ vertices, each with a thickness of $2$ (possibly $\ell = 0$). Above this base, there are $4$ possible types.

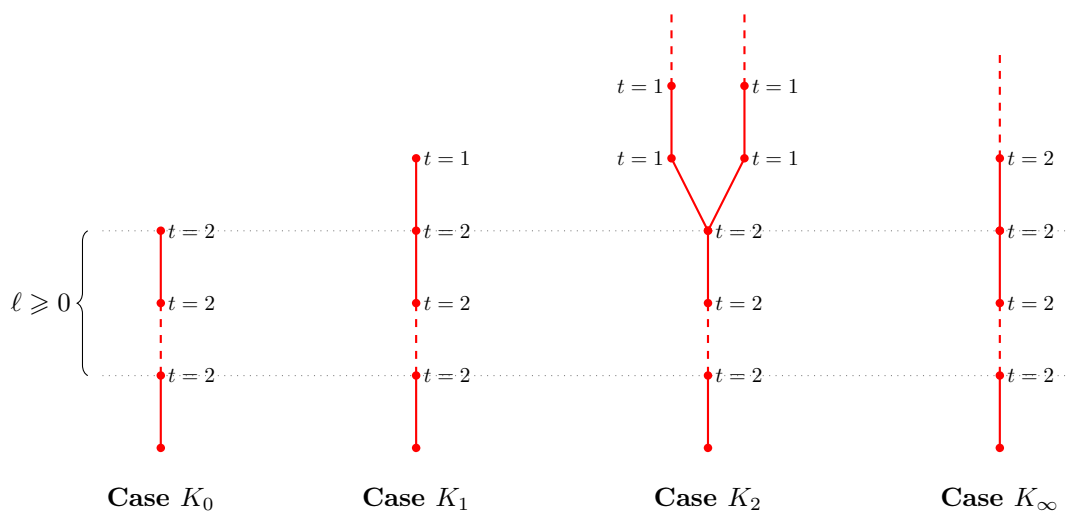
\begin{figure}[H]
	\myfigure{0.8}{
\begin{tikzpicture}[scale=1.2]
\usetikzlibrary{decorations.pathreplacing}
\tikzset{
  line/.style = {
  },
  vector/.style = {
    thick,-latex
  },
  dot/.style = {
    insert path={
      node[scale=3]{.}
    }
  }
}

\draw [decorate,decoration={brace,amplitude=4pt,mirror},xshift=0.5cm,yshift=0pt]
      (-1,0) -- (-1,-2) node [midway,left,xshift=-.1cm] {$\ell\ge0$};

\draw[gray,dotted,thin] (-0.3,0) -- (13,0);
\draw[gray,dotted,thin] (-0.3,-2) -- (13,-2);

\begin{scope}[xshift=0.5cm]

\path 
  (0,0) edge[line, thick,red, ] (0,-1)
  (0,-1) edge[line, thick,red, dashed] (0,-2)
  (0,-2) edge[line, thick,red] (0,-3)
;

 \path
   (0,0) [red,dot] node[right,black,scale=0.8]{$t=2$}
   (0,-1) [dot] node[right,black,scale=0.8]{$t=2$}
   (0,-2) [dot] node[right,black,scale=0.8]{$t=2$}
   (0,-3) [dot] node[right,black,scale=0.8]{}
;

\node at (0,-3.7) {\bf Case $K_0$};
\end{scope}

\begin{scope}[xshift=4cm]

\path 
  (0,0) edge[line, thick,red, ] (0,-1)
  (0,-1) edge[line, thick,red, dashed] (0,-2)
  (0,-2) edge[line, thick,red] (0,-3)
;

 \path
   (0,0) [red,dot] node[right,black,scale=0.8]{$t=2$}
   (0,-1) [dot] node[right,black,scale=0.8]{$t=2$}
   (0,-2) [dot] node[right,black,scale=0.8]{$t=2$}
   (0,-3) [dot] node[right,black,scale=0.8]{}
;

\path  (0,0) edge[line, thick,red, ] (0,1);
 \path (0,1) [red,dot] node[right,black,scale=0.8]{$t=1$};

\node at (0,-3.7) {\bf Case $K_1$};
\end{scope}

\begin{scope}[xshift=8cm]

\path 
  (0,0) edge[line, thick,red, ] (0,-1)
  (0,-1) edge[line, thick,red, dashed] (0,-2)
  (0,-2) edge[line, thick,red] (0,-3)
;

 \path
   (0,0) [red,dot] node[right,black,scale=0.8]{$t=2$}
   (0,-1) [dot] node[right,black,scale=0.8]{$t=2$}
   (0,-2) [dot] node[right,black,scale=0.8]{$t=2$}
   (0,-3) [dot] node[right,black,scale=0.8]{}
;

 \path
   (0,0) coordinate (P0)
   (-0.5,1) coordinate (P11)
   (-0.5,2) coordinate (P12)
   (0.5,1) coordinate (P21)
   (0.5,2) coordinate (P22)

 ;

\path 
  (P0) edge[line,  thick, red] (P11)
  (P11) edge[line,  thick, red] (P12)
  (P12) edge[line,  thick, dashed, red] (-0.5,3)
  (P0) edge[line,  thick, red] (P21)
  (P21) edge[line,  thick, red] (P22)
  (P22) edge[line,  thick, dashed, red] (0.5,3)
;

 \path
   (P0) [red,dot,] node[below]{}
   (P11) [dot, ] node[black,left, scale=0.8]{$t=1$}
   (P12) [dot, ] node[black,left, scale=0.8]{$t=1$}

   (P21) [dot, ] node[black,right, scale=0.8]{$t=1$}
   (P22) [dot, ] node[black,right, scale=0.8]{$t=1$}

 ;
\node at (0,-3.7) {\bf Case $K_2$};
\end{scope}

\begin{scope}[xshift=12cm]

\path 
  (0,0) edge[line, thick,red, ] (0,-1)
  (0,-1) edge[line, thick,red, dashed] (0,-2)
  (0,-2) edge[line, thick,red] (0,-3)
;

 \path
   (0,0) [red,dot] node[right,black,scale=0.8]{$t=2$}
   (0,-1) [dot] node[right,black,scale=0.8]{$t=2$}
   (0,-2) [dot] node[right,black,scale=0.8]{$t=2$}
   (0,-3) [dot] node[right,black,scale=0.8]{}
;

\path 
  (0,0) edge[line, thick,red, ] (0,1)
  (0,1) edge[line, thick,red, , dashed] (0,2.5)
;

 \path
   (0,0) [red,dot] {}
   (0,1) [dot] node[right,black,scale=0.8]{$t=2$}

;

\node at (0,-3.7) {\bf Case $K_\infty$};
\end{scope}

\end{tikzpicture}%

	} 
	\caption{Possible trunks.}
	\label{fig:trunkdeg2}
\end{figure}

\medskip

\textbf{Base.}

Since $P$ is a polynomial of degree $2$, the thickness is at most $2$, and this is only possible for a single root; the same holds for its successors. What can happen after a stem with thicknesses only $2$?

\smallskip

\textbf{Cases $K_0$ and $K_\infty$.}
This base stem of thickness $2$ can be infinite; this is the case, for example, for $P(X) = X^2$.  
The trunk can also stop just after the last vertex of thickness $2$ in the base.  
This is, for example, the case for $P(X) = (X-1)^2 + p^{2\ell}$ with $p=3$.

\smallskip

\textbf{Case $K_2$.}
In the remaining cases each vertex just after the base has a thickness of $1$.  
Let us then consider the polynomial $Q$ associated with the last vertex of thickness $2$. Suppose it has two simple roots modulo $p$ (that is $Q(x_i) \equiv 0 \pmod{p}$ and $Q'(x_i) \not\equiv 0 \pmod{p}$, $i=1,2$). Hensel's lemma then allows these two simple roots to be ``lifted'' indefinitely: for any $e \ge 1$, there exists $\tilde{x_i} \in \Zz$ ($i=1,2$) such that $\tilde{x_i} \equiv x_i \pmod{p}$ and $Q(\tilde{x_i}) \equiv 0 \pmod{p^e}$.
This is then the $K_2$ situation.

\begin{example}
	This is the case for $P(X) = (X-1)(X-2)+p$ with, for example, $p=5$.
	For $e=1$, the solutions of $P(x) \equiv 0 \pmod{p^e}$ are $\{1,2\}$.
	For $e=2$, it's $\{6,22\}$, and for $e=3$, it's $\{31,97\}$\ldots{}
	In this example the base is empty.
	
	\begin{figure}[H]
		\myfigure{0.6}{
\begin{tikzpicture}[scale=1.5]

\tikzset{
  line/.style = {
  },
  vector/.style = {
    thick,-latex
  },
  dot/.style = {
    insert path={
      node[scale=3]{.}
    }
  }
}

\begin{scope}

 \path
   (0,0) coordinate (P0)
   (-0.5,1) coordinate (P11)
   (-0.5,2) coordinate (P12)
   (-0.5,3) coordinate (P13)
   (-0.5,4) coordinate (P14)
   (0.5,1) coordinate (P21)
   (0.5,2) coordinate (P22)
   (0.5,3) coordinate (P23)
   (0.5,4) coordinate (P24)

 ;

\path 
  (P0) edge[line,  thick, red] (P11)
  (P11) edge[line,  thick, red] (P12)
  (P12) edge[line,  thick, red] (P13)
  (P13) edge[line,  thick, red] (P14)
  (P14) edge[line,  thick, dashed, red] (-0.5,5)
  (P0) edge[line,  thick, red] (P21)
  (P21) edge[line,  thick, red] (P22)
  (P22) edge[line,  thick, red] (P23)
  (P23) edge[line,  thick, red] (P24)
  (P24) edge[line,  thick, dashed, red] (0.5,5)
;


 \path
   (P0) [red,dot, red] node[below]{}
   (P11) [dot, red] node[black,left, scale=0.8]{$1$}
   (P12) [dot, red] node[black,left, scale=0.8]{$6$}
   (P13) [dot, red] node[black,left, scale=0.8]{$31$}
   (P14) [dot, red] node[black,left, scale=0.8]{$281$}

   (P21) [dot, red] node[black,right, scale=0.8]{$2$}
   (P22) [dot, red] node[black,right, scale=0.8]{$22$}
   (P23) [dot, red] node[black,right, scale=0.8]{$97$}
   (P24) [dot, red] node[black,right, scale=0.8]{$347$}

 ;


\end{scope}

\end{tikzpicture}%

		} 
		\caption{The trunk (and the tree) of $P(X)=(X-1)(X-2)+5$.}
		\label{fig:simpleroot}
	\end{figure}
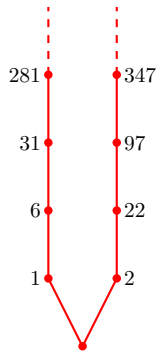

	More broadly, the polynomial $P(X) = (X-p^\ell)(X-2p^\ell) + p^{2\ell + 1}$ has a trunk, as in the case $K_2$, with a base of $\ell$ vertices and thickness $2$. 
	
\end{example}

\smallskip

\textbf{Case $K_1$.}
Let us resume the discussion started in the previous case. It is possible for $Q$ to have a double root of thickness $1$, as in the case of $Q(X) = (X - x_0)^2 + p$.  
We then have $Q(x_0 + pX) = p \big( pX^2 + 1 \big)$,
which indeed gives a thickness of $1$, but the successor of $Q$ is $pX^2 + 1$, which has no root modulo $p$.  
Thus, the trunk ends here in the $K_1$ configuration. This is, for example, the case for  
$P(X) = (X - 1)^2 + p^{2\ell+1}$ with $p = 3$.

\begin{example}
	Let $P(X) = (X-1)^2+p^5$ with $p=3$.
	Here are the solutions of the equation $P(x)\equiv 0 \pmod{p^e}$ for different values of $e$:
	\[
	\begin{array}{cc}
		p^e & \text{ solutions } \\ \hline
		3^1 & 1 \\
		3^2 & 1,4,7 \\
		3^3 & 1,10,19 \\	
		3^4 & 1, 10, 19, 28, 37, 46, 55, 64, 73 \\
		3^5 & 1, 28, 55, 82, 109, 136, 163, 190, 217 \\
	\end{array}
	\]
	
	For $e \ge 6$, the equation has no solutions.

	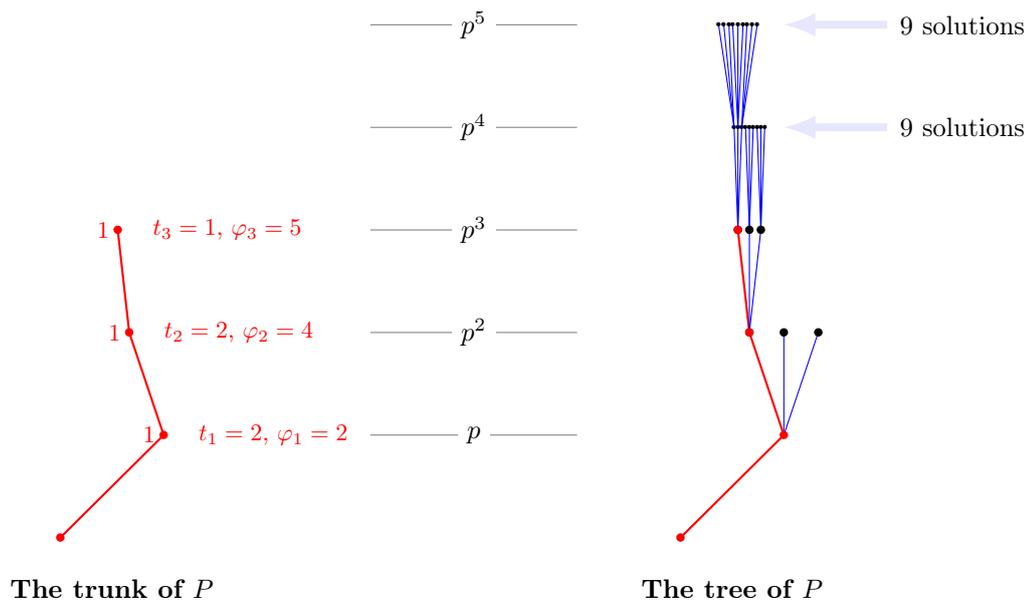
\begin{figure}[H]
		\myfigure{0.8}{
\begin{tikzpicture}[scale=1.7]

\tikzset{
  line/.style = {
  },
  vector/.style = {
    thick,-latex
  },
  dot/.style = {
    insert path={
      node[scale=3]{.}
    }
  },
  smalldot/.style = {
    insert path={
      node[scale=1.5]{.}
    }
  }
}

\begin{scope}
 \path
   (0,0) coordinate (O)
 ;

\foreach \i in {0,...,2} {
     \path (\i-1,1) coordinate (P\i1);
  }

\foreach \i in {0,...,8} {
     \path ({(\i-4)/3},2) coordinate (P\i2);
  }
\foreach \i in {0,...,26} {
     \path ({(\i-13)/9},3) coordinate (P\i3);
  }
\foreach \i in {0,...,80} {
     \path ({(\i-40)/27},4) coordinate (P\i4);
  }

\path (O) edge[line,  thick, red] (P21);
 

\path (P21) edge[line,  thick, red] (P62);

\path (P62) edge[line,  thick, red] (P183);


 \path
   (O) [red, dot, red] node[below]{}
   (P21) [dot, red] node[left,scale=0.9]{$1$} node[right=1em,scale=0.9]{$t_1 = 2$, $\varphi_1 = 2$}
   (P62) [dot, red] node[left,scale=0.9]{$1$} node[right=1em,scale=0.9]{$t_2 = 2$, $\varphi_2 = 4$}
   (P183) [dot, red] node[left,scale=0.9]{$1$} node[right=1em,scale=0.9]{$t_3 = 1$, $\varphi_3 = 5$}
 ;

\node at (0.5,-0.5) {\bf The trunk of $P$};

  \draw[line,thin,gray] (3,1) -- ++(2,0) node[midway, black, fill=white,]{$p$};
  \draw[line,thin,gray] (3,2) -- ++(2,0) node[midway, black, fill=white, ]{$p^2$};
  \draw[line,thin,gray] (3,3) -- ++(2,0) node[midway, black, fill=white, ]{$p^3$};
  \draw[line,thin,gray] (3,4) -- ++(2,0) node[midway, black, fill=white, ]{$p^4$};
  \draw[line,thin,gray] (3,5) -- ++(2,0) node[midway, black, fill=white, ]{$p^5$};

\end{scope}

\begin{scope}[xshift=6cm]
 \path
   (0,0) coordinate (O)
 ;

\foreach \i in {0,...,2} {
     \path (\i-1,1) coordinate (P\i1);
  }

\foreach \i in {0,...,8} {
     \path ({(\i-4)/3},2) coordinate (P\i2);
  }
\foreach \i in {0,...,26} {
     \path ({(\i-13)/9},3) coordinate (P\i3);
  }
\foreach \i in {0,...,80} {
     \path ({(\i-40)/27},4) coordinate (P\i4);
  }
\foreach \i in {0,...,80} {
     \path ({(\i-40)/27},4) coordinate (P\i4);
  }

\path (O) edge[line,  thick, red] (P21);
 

\path (P21) edge[line,  thick, red] (P62);

\path (P62) edge[line,  thick, red] (P183);

\path (P21) edge[line, blue] (P72);
\path (P21) edge[line, blue] (P82);

\path (P62) edge[line, blue] (P193);
\path (P62) edge[line, blue] (P203);

\path (P183) edge[line, thin, blue] (P544);
\path (P183) edge[line, thin, blue] (P554);
\path (P183) edge[line, thin, blue] (P564);

\path (P193) edge[line, thin, blue] (P574);
\path (P193) edge[line, thin, blue] (P584);
\path (P193) edge[line, thin, blue] (P594);

\path (P203) edge[line, thin, blue] (P604);
\path (P203) edge[line, thin, blue] (P614);
\path (P203) edge[line, thin, blue] (P624);


  \foreach \i in {0,...,0} {
  }
  \foreach \i in {6,...,8} {
     \path  (P\i2) [dot] {};
  }
  \foreach \i in {18,...,20} {
     \path  (P\i3) [dot] {};
  }
  \foreach \i in {54,...,62} {
     \path  (P\i4) [smalldot] {};
  }

  \draw[line, thin, blue,](P544) -- ++ (-0.15,1) coordinate(Q1);
  \draw[line, thin, blue,](P544) -- ++ (-0.1,1) coordinate(Q2);
  \draw[line, thin, blue,](P544) -- ++ (-0.05,1) coordinate(Q3);

  \draw[line, thin, blue,](P554) -- ++ (-0.05,1) coordinate(Q4);
  \draw[line, thin, blue,](P554) -- ++ (-0.,1) coordinate(Q5);
  \draw[line, thin, blue,](P554) -- ++ (0.05,1) coordinate(Q6);

  \draw[line, thin, blue,](P564) -- ++ (0.05,1) coordinate(Q7);
  \draw[line, thin, blue,](P564) -- ++ (0.1,1) coordinate(Q8);
  \draw[line, thin, blue,](P564) -- ++ (0.15,1) coordinate(Q9);

  \foreach \i in {1,...,9} {
     \path  (Q\i) [smalldot] {};
  }

 \path
   (O) [red, dot, red] {}
   (P21) [dot, red] {}
   (P62) [dot, red] {} 
   (P183) [dot, red] {} 
 ;

  \draw[<-,>=latex,line width=3pt,blue!10] (1,4) -- ++(1,0) node[right, black]{$9$ solutions};
  \draw[<-,>=latex,line width=3pt,blue!10] (1,5) -- ++(1,0) node[right, black]{$9$ solutions};

\node at (0.5,-0.5) {\bf The tree of $P$};

\end{scope}

\end{tikzpicture}%

		} 
		\caption{The trunk and the tree of $P(X)=(X-1)^2+3^5$.}
		\label{fig:trunkdouble}
	\end{figure}
	
\end{example}

\section{Perspectives and references}

\subsection{Perspectives}


Let us conclude by discussing the \emph{Poincaré series}:  
\[
S(u) = \sum_{e \geq 0} N_e \frac{u^e}{p^e},
\]  
where $N_e$ denotes the number of solutions to the equation $P(x) \equiv 0 \pmod{p^e}$  
(with the convention that $N_0 = 1$).  
The series $S(u)$ serves as the natural generating function associated with the number of solutions. In fact, it is a rational function in $u$:
\begin{theorem}[Igusa]
	\label{th:igusa}
	\[
	S(u) \in \Qq(u)
	\]
\end{theorem}

We leave it to the reader to prove this result by relying on the structure of the trunk and Corollary \ref{cor:count}.  
For polynomials in several variables, an analogous result, due to Igusa, remains valid. This area of research is still very active today \cite{PoVe}.

\subsection{References}
\label{ssec:ref}

Our problem is masterfully addressed by Schmidt and Stewart in 1997, in the article \cite{ScSt} which utilizes graph studies and contains, either explicitly or implicitly, all the notions and results of the present article as well as numerous additional results. It seems that this article did not receive the widespread attention it deserved.

Fortunately, given its importance, the problem and the solution have resurfaced multiple times, especially when it comes to finding algorithmic solutions to polynomial problems. Thus, the explicit construction of the central notion of the present article, that of the ``trunk'', is given by Zúñiga-Galindo \cite{ZG2003} for a calculation of Igusa’s zeta function. The same construction is found in the article by Berthomieu, Lecerf, and Quintin \cite{BLQ2013} for determining the roots of polynomials in local rings. These same objects and results are taken up by Dwivedi, Mittal, and Saxena \cite{DMS2019}, \cite{DMS2019P4}, \cite{DS2020}, for example, for factorization problems. Finally, Kopp, Randall, Rojas, and Zhu \cite{Kopp} define, draw, and use the trunk to count the number of solutions without explicitly detailing them.

Regarding the more elementary notion of the tree of solutions modulo $p^e$ of a polynomial, classic references are \cite{Apo1976} or \cite{NZM}.

\bigskip

\emph{Acknowledgments.}
	We thank the referees and the editors for their helpful comments and suggestions.


\bibliographystyle{plain}
\bibliography{tree.bib}

\end{document}